\documentclass[11pt,reqno]{amsart}
\usepackage{amsmath,amssymb}

 \makeatletter
 \evensidemargin  \oddsidemargin
 \makeatother

\begin{document}
\thispagestyle{empty}
 \setcounter{page}{1}

{\Large \bf  \begin{center} Stability of cubic and quartic
functional equations in non-Archimedean spaces
\end{center} }

\vskip.20in
\begin{center}

\vskip.30in

{\Small\bf  M. Bavand Savadkouhi and M. Eshaghi Gordji } \\[2mm]

{\footnotesize Department of Mathematics,
Semnan University,\\ P. O. Box 35195-363, Semnan, Iran\\
[-1mm] e-mail: {\tt bavand.m@gmail.com \& madjid.eshaghi@gmail.com}}

\end{center}
\vskip 5mm \noindent{\footnotesize{\bf Abstract.} We prove
generalized Hyres-Ulam-Rassias stability of the cubic functional
equation $f(kx+y)+f(kx-y)=k[f(x+y)+f(x-y)]+2(k^3-k)f(x)$ for all
$k\in \Bbb N$ and the quartic functional equation
$f(kx+y)+f(kx-y)=k^2[f(x+y)+f(x-y)]+2k^2(k^2-1)f(x)-2(k^2-1)f(y)$
for all $k\in \Bbb N$ in non-Archimedean normed spaces. \vskip.10in
\footnotetext {2000 Mathematics Subject Classification. 39B22,
39B82, 46S10.} \footnotetext {Keywords: Generalized
Hyers-Ulam-Rassias stability; Cubic functional equation; Quartic
functional equation; Non-Archimedean space; p-adic.} \vskip.10in

  \newtheorem{df}{Definition}[section]
  \newtheorem{rk}[df]{Remark}
   \newtheorem{lem}[df]{Lemma}
   \newtheorem{thm}[df]{Theorem}
   \newtheorem{pro}[df]{Proposition}
   \newtheorem{cor}[df]{Corollary}
   \newtheorem{ex}[df]{Example}

 \setcounter{section}{0}
 \numberwithin{equation}{section}

\vskip .2in

\begin{center}
\section{Introduction}
\end{center}
A classical question in the theory of functional equations is the
following: "When is it true that a function which approximately
satisfies a functional equation $\epsilon$ must be close to an
exact solution of $\epsilon ?$ "\\ If the problem accepts a
solution, we say that the equation $\epsilon$ is stable. The first
stability problem concerning group homomorphisms was raised by
Ulam \cite{Ul} in $1940.$\\
We are given a group $G$ and a metric group $G^{'}$ with metric
$d(.,.).$ Given $\epsilon>0,$ does there exist a $\delta>0$ such
that if $f:G \to G^{'}$ satisfies $d(f(xy),f(x)f(y))< \delta$ for
all $x,y \in G,$ then a homomorphism $h:G \to G^{'}$ exists with
$d(f(x),h(x))< \epsilon$ for all $x\in G?$\\
Ulam's problem was partially solved by Hyers \cite{Hy} in $1941.$
Let $E_1$ be a normed space, $E_2$ a Banach space and suppose that
the mapping $f:E_1 \to E_2$ satisfies the inequality $$\|f(x+y)-
f(x)-f(y)\| \leq \epsilon \hspace{2cm}(x,y \in E_1),$$ where
$\epsilon>0$ is a constant. Then the limit $T(x)=\lim_{n \to
\infty}2^{-n}f(2^nx)$ exists for each $x \in E_1$ and $T$ is the
unique additive mapping satisfying  $$\|f(x)-T(x)\|\leq \epsilon
\eqno(1.1)$$ for all $x \in E_1.$ Also, if for each $x$ the function
$t \longmapsto f(tx)$ from $\Bbb R$ to $E_2$ is continuous on $\Bbb
R,$ then $T$ is linear. If $f$ is continuous at a single point of
$E_1,$ then $T$ is continuous everywhere in $E_1.$
Moreover $(1.1)$ is sharp.\\
In $1978,$ Th. M. Rassias \cite{TRa1} formulated and proved the
following theorem, which implies Hyers' theorem as a special case.
Suppose that $E$ and $F$ are real normed spaces with $F$ a complete
normed space, $f:E \to F$ is a mapping such that for each fixed $x
\in E$ the mapping $t \longmapsto f(tx)$ is continuous on $\Bbb R,$
and let there exist $\epsilon > 0$ and $p \in [0, 1)$ such that
$$\|f(x+y)-f(x)-f(y)\| \leq \epsilon (\|x\|^p+\|y\|^p)\eqno(1.2)$$
for all $x,y \in E.$ Then there exists a unique linear mapping $T:E
\to F$ such that
$$\|f(x)-T(x)\| \leq \frac{\epsilon \|x\|^p}{(1 - 2^{p-1})}$$ for all $x \in E.$ The case of the
existence of a unique additive mapping had been obtained by $T.$ The
terminology Hyers-Ulam stability originates from these historical
backgrounds. The terminology can also be applied to the case of
other functional equations. For more detailed definitions of such
terminologies, we can refer to \cite{Ga},\cite{Gr},\cite{Hy-Is-Ra}
and \cite{Ju}. In $1994,$ P. G\v avruta \cite{G} provided a further
generalization of Th. M. Rassias' theorem in which he replaced the
bound $\epsilon(\|x\|^p +\|y\|^p)$ in $(1.2)$ by a general control
function $\varphi(x,y)$
for the existence of a unique linear mapping.\\
The functional equation $f(x+y)+f(x-y)=2f(x)+2f(y)$ is called the
quadratic functional equation. In particular, every solution of the
quadratic functional equation is said to be a quadratic mapping, see
\cite{R1, TRa2}. A generalized Hyers-–Ulam stability problem for the
quadratic functional equation was proved by Skof \cite{Sk} for
mappings $f:X \to Y,$ where $X$ is a normed space and $Y$ is a
Banach space. Cholewa \cite{Ch} noticed that the theorem of Skof is
still true if the relevant domain $X$ is replaced by an Abelian
group. In \cite{Cz}, Czerwik proved the generalized Hyers–-Ulam
stability of the quadratic functional equation. Borelli and Forti
\cite{Bo-Fo} generalized the stability result as follows (cf.
\cite{Pa1,Pa2}): Let $G$ be an Abelian group, and $X$ a Banach
space. Assume that a mapping $f:G \to X$ satisfies the functional
inequality $$\|f(x+y)+f(x-y) -2f(x)-2f(y)\| \leq \varphi(x,y)$$ for
all $x,y \in G,$ and $\varphi:G \times G \to [0,\infty)$ is a
function such that
$$\Phi(x,y):=\sum_{i=0}^{\infty}\frac{1}{4^{i+1}} \varphi(2^ix,2^iy)< \infty$$ for all $x,y \in G.$ Then
there exists a unique quadratic mapping $Q: G \to X$ with the
property  $\|f(x)-Q(x)\| \leq \Phi(x,x)$ for all $x \in G.$\\ Jun
and Kim \cite{Ju-Ki} introduced the following cubic functional
equation
$$f(2x+y)+f(2x-y)=2f(x+y)+2f(x-y)+12f(x)\eqno (1.3)$$ and they
established the general solution and the generalized
Hyers-Ulam-Rassias stability for the functional equation $(1.3).$
The function $f(x)=x^3$ satisfies the functional equation $(1.3),$
which is thus called a cubic functional equation. Every solution
of the cubic functional equation is said to be a cubic function.
Now we introduce the following cubic functional equation
$$f(kx+y)+f(kx-y)=k[f(x+y)+f(x-y)]+2(k^3-k)f(x) \eqno(1.4)$$ for
all $k\in \Bbb N.$ \\ In \cite{Pa-Ba}, Won-Gil Park and Jae-Hyeong
Bae, introduced the following functional equation
$$f(x+2y)+ f(x-2y)=4[[f(x+y)+f(x-y)]+24f(y)-6f(x)\eqno (1.5)$$ and
they established the general solution of the functional equation
$(1.5).$ It is easy to see that the function $f(x)=cx^4$ is a
solution of the functional equation $(1.5).$ Thus, it is natural
that $(1.5)$ is called a quartic functional equation and every
solution of the quartic functional equation is said to be a quartic
mapping. In this paper, we introduce the following quartic
functional equation
$$f(kx+y)+f(kx-y)=k^2[f(x+y)+f(x-y)]+2k^2(k^2-1)f(x)-2(k^2-1)f(y) \eqno(1.6)$$
for all $k\in \Bbb N$. For more detailed definitions of such
terminologies, we can refer to \cite{Na}.\\
By a non-Archimedean field we mean a field $K$ equipped with a
function (valuation) $|.|$ from $K$ into $[0,\infty)$ such that
$|r|= 0$ if and only if $r=0,$ $|rs|=|r||s|,$ and $|r+s| \leq \max
\{|r|, |s|\}$ for all $r,s \in K.$ Clearly $|1|=|-1|=1$ and $|n|
\leq 1$ for all $n \in N.$\\
\begin{df}Let $X$ be a vector space over a scalar field $K$ with a
non-Archimedean non-trivial valuation $|.|.$ A function $\|.\|:X
\to \Bbb R$ is a non-Archimedean norm (valuation) if it satisfies
the following conditions:\\ $(i)$ $\|x\|=0$ if and only if
$x=0;$\\
$(ii)$ $\|rx\|=|r| \|x\| $ for all $r \in  K,$ $x \in X;$\\
$(iii)$ the strong triangle inequality (ultrametric); namely,
$$\|x+y\| \leq \max \{ \|x\|,\|y\| \}.$$
for all $x,y \in X.$ Then $(X,\|.\|)$ is called a non-Archimedean
space.
\end{df}
Due to the fact that $$\|x_n-x_m\| \leq \max \{\|x_{j+1}-x_{j}\| :
m \leq j \leq n-1 \} \hspace{1cm} (n> m)$$ a sequence $\{x_n\}$ is
Cauchy if and only if $\{x_{n+1}-x_n \}$ converges to zero in a
non-Archimedean space. By a complete non-Archimedean space we mean
one in which every Cauchy sequence is convergent.\\
In $1897,$ Hensel  \cite{He} discovered the p-adic numbers as a
number theoretical analogue of power series in complex analysis.
Fix a prime number $p.$ For any nonzero rational number $x,$ there
exists a unique integer $n_x \in \Bbb Z$ such that $x=\frac{a}{b}
p^{n_x},$ where $a$ and $b$ are integers not divisible by $p.$
Then $|x|_p:=p^{-n_x}$ defines a non-Archimedean norm on $\Bbb Q.$
The completion of $\Bbb Q$ with respect to the metric
$d(x,y)=|x-y|_p$ is denoted by $\Bbb Q_p$ which is called the
p-adic number field. In fact, $\Bbb Q_p$ is the set of all formal
series $x=\sum_{k \geq n_x}^{\infty}a_kp^k,$ where $|a_k|\leq p-1$
are integers. The addition and multiplication between any two
elements of $\Bbb Q_p$ are defined naturally. The norm $|\sum_{k
\geq n_x}^{\infty}a_kp^k|_p=p^{-n_x}$ is a non-Archimedean norm on
$\Bbb Q_p$ and it makes $\Bbb Q_p$ a locally compact field; see
\cite{Go,Ro}.\\
During the last three decades p-adic numbers have gained the
interest of physicists for their research in particular in problems
coming from quantum physics, p-adic strings and superstrings (cf.
\cite{Kh}). A key property of p-adic numbers is that they do not
satisfy the Archimedean axiom: for all $x,y > 0,$ there exists an
integer $n$ such that $x< ny.$ "It is very difficult to imagine a
situation where this axiom does not hold, but in fact the very space
and time we inhabit have both been shown by $20$th century science
to be unequivocally non-Archimedean: The Archimedean axiom breaks
down at the Planck scale, that is, for distances less than $1.6
\times 10^{-33}$ metres and durations less than $5.4 \times
10^{-44}$ seconds. Despite our entrenched belief that space and time
are continuous, homogeneous, infinitely divisible quantities, we are
now confronted with the fact that below this scale, distances and
durations cannot scaled up in order to produce macroscopic distances
and durations. Equivalently, we cannot meaningfully measure
distances or durations below this scale. So a suggestion emerges to
abandon the Archimedean axiom at very small distances. This leads to
a non-Euclidean and non-Riemannian geometry of space
at small distances"; cf. \cite{Vl-Vo-Ze}.\\
In \cite{Ar-Be}, the authors investigated stability of approximate
additive mappings $f:\Bbb Q_p \to \Bbb R.$ They showed that if $f:
\Bbb Q_p \to \Bbb R$ is a continuous mapping for which there
exists a fixed $\epsilon$ such that $|f(x+y)-f(x)-f(y)| \leq
\epsilon$ for all $x,y \in Q_p,$ then there exists a unique
additive mapping $T:\Bbb Q_p \to \Bbb R$ such that
$|f(x)-T(x)|\leq \epsilon$ for all $x \in \Bbb Q_p.$\\
M. S. Moslehian and Th. M. Rassias \cite{Mo-Ra} proved the
generalized Hyers-Ulam stability of the Cauchy functional equation
and the quadratic functional equation in non-Archimedean normed
spaces.\\
In this paper, we solve the stability problem for cubic and quartic
functional equations when the unknown function is one with values in
a non-Archimedean space, in particular in the field of p-adic
numbers. \vskip .2in
\section {{\Large\bf} Stability of the cubic functional equation}
In this section, we prove the generalized Hyers-Ulam-Rassias
stability of the cubic functional equation $(1.4)$. Throughout this
section, we assume that $G$ is an additive group and $X$ is a
complete non-Archimedean space. Now before taking up the main
subject, given $f:G \to X,$ we define the difference operator
$$Df(x,y)=f(kx+y)+f(kx-y)-k[f(x+y)+f(x-y)]-2(k^3-k)f(x)$$ for all
$x,y \in G$ and all $k \in \Bbb N.$ We consider the following
functional inequality:
$$\|Df(x,y)\| \leq \varphi(x,y)$$
for an upper bound $\varphi:G \times G \to [0, \infty).$
\begin{thm}
Let $\varphi:G \times G \to [0,\infty)$ be a function such that
$$\lim_{n \to \infty} \frac{\varphi(k^nx,k^ny)}{|k|^{3n}}=0 \eqno(2.1)$$
for all $x,y \in G,$ $k \in \Bbb N$ and let for each $x \in G$ the
limit
$$\lim_{n \to \infty} \max \{\frac{\varphi(k^jx,0)}{|k|^{3j}}:~0 \leq j < n \},\eqno(2.2)$$
denoted by $\tilde{\varphi}(x),$ exists. Suppose that $f:G \to X$
is a mapping satisfying
$$\|Df(x,y)\| \leq \varphi(x,y) \eqno(2.3)$$
for all $x,y \in G.$ Then there exists a cubic mapping $C:G \to X$
such that
$$\|C(x)-f(x)\| \leq \frac{1}{|2k^3|} \tilde{\varphi}(x) \eqno(2.4)$$
for all $x \in G,$ if $$\lim_{i \to \infty} \lim_{n \to \infty}
\max \{\frac{\varphi(k^jx,0)}{|k|^{3j}}:~i \leq j < n+i \}=0$$
then $C$ is the unique cubic mapping satisfying $(2.4).$
\end{thm}
\begin{proof}
Putting $y=0$ in $(2.3),$ we get
$$\|\frac{1}{k^3}f(kx)-f(x)\| \leq \frac{1}{|2k^3|} \varphi(x,0)\eqno(2.5)$$
for all $x \in G$ and all $k \in \Bbb N.$ Replacing $x$ by
$k^{n-1}x$ in $(2.5),$ we get
$$\|\frac{1}{k^{3n}}f(k^nx)-\frac{1}{k^{3(n-1)}}f(k^{n-1}x)\| \leq \frac{\varphi(k^{n-1}x,0)}{|2k^{3n}|}\eqno(2.6)$$
for all $x \in G$ and all $k \in \Bbb N.$ It follows from $(2.6)$
and $(2.1)$ that the sequence $\{ \frac{f(k^nx)}{k^{3n}} \}$ is
Cauchy. Since $X$ is complete, we conclude that $\{
\frac{f(k^nx)}{k^{3n}} \}$ is convergent. Set $C(x):=\lim_{n \to
\infty} \frac{f(k^nx)}{k^{3n}}.$ Using induction one can show that
$$\| \frac{f(k^nx)}{k^{3n}}-f(x)\| \leq  \frac{1}{|2k^3|} \max
\{\frac{\varphi(k^ix,0)}{|k|^{3i}}:~ 0 \leq i < n \} \eqno(2.7)$$
for all $n \in \Bbb N$ and all $x \in G.$ By taking $n$ to approach
infinity in $(2.7)$ and using $(2.2)$, we  obtain $(2.4).$ By
$(2.1)$ and $(2.3),$ we get
$$\|DC(x,y)\|=\lim_{n \to \infty}\frac{1}{|k^{3n}|}\|f(k^nx,k^ny)\| \leq \lim_{n \to \infty} \frac{\varphi(k^nx,k^ny)}{|k|^{3n}}=0$$
for all $x,y \in G$ and all $k \in \Bbb N.$ Therefore the mapping
$C:G \to X$ satisfies $(1.4).$ If $C^{'}$ is another cubic mapping
satisfying $(2.4),$ then
\begin{align*}
\|C(x)-C^{'}(x)\|&=\lim_{i \to
\infty}|k|^{-3i}\|C(k^ix)-C^{'}(k^ix)\|\\
& \leq \lim_{i \to \infty} |k|^{-3i} \max \{~
\|C(k^ix)-f(k^ix)\|,\|f(k^ix)-C^{'}(k^ix)\|~ \}\\
&\leq \frac{1}{|2k^3|} \lim_{i \to \infty} \lim_{n \to \infty}
\max \{ \frac{\varphi(k^jx,0)}{|k|^{3j}}:~i \leq j < n+i \}\\
&=0.
\end{align*}
for all $x \in G$ and all $k \in \Bbb N.$ Therefore $C=C^{'}.$
This completes the proof of the uniqueness of $C.$
\end{proof}
\begin{cor}
Let $\alpha:[0,\infty) \to [0,\infty)$ be a function satisfying\\
$(i)$ $\alpha(|k|t) \leq \alpha(|k|) \alpha(t)$ for all $t \geq
0,$\\ $(ii)$ $\alpha(|k|) < |k|^3.$\\ Let $\delta > 0,$ let $G$ be
a normed space and let $f:G \to X$ fulfill the inequality
$$\|Df(x,y)\| \leq \delta [\alpha (\|x\|)+\alpha (\|y\|)]$$ for all
$x,y \in G.$ Then there exists a unique cubic mapping $C:G \to X$
such that
$$\|f(x)-C(x)\| \leq \frac{1}{|2k^3|} \delta \alpha (\|x\|) \eqno(2.8)$$ for all $x \in G$ and all $k \in \Bbb N.$
\end{cor}
\begin{proof}
Defining $\varphi:G \times G \to [0,\infty)$ by
$\varphi(x,y):=\delta [\alpha (\|x\|)+\alpha (\|y\|)]$ we have
$$\lim_{n \to \infty} \frac{\varphi(k^nx,k^ny)}{|k|^{3n}} \leq \lim_{n \to \infty}(\frac{\alpha(|k|)}{|k|^3})^{n} \varphi(x,y)=0$$
for all $x,y \in G$ and all $k \in \Bbb N.$ We have
$$\tilde{\varphi}(x)=\lim_{n \to \infty} \max \{ \frac{\varphi(k^jx,0)}{|k|^{3j}}:~0 \leq j < n
\}=\varphi(x,0)$$ and
$$\lim_{i \to \infty} \lim_{n \to \infty} \max \{ \frac{\varphi(k^jx,0)}{|k|^{3j}}:~ i \leq j < n+i \}= \lim_{i \to \infty} \frac{\varphi(k^ix,0)}{|k|^{3i}}=0$$
for all $x \in G$ and all $k \in \Bbb N.$
\end{proof}
\begin{rk}
The classical example of the function $\alpha$ is the mapping
$\alpha(t)=t^p$ for all $t \in [0, \infty),$ where $p>3$ with the
further assumption that $|k|<1.$
\end{rk}
\begin{rk}
We can formulate similar statements to Theorem $2.1$ in which we
can define the sequence $C(x):=\lim_{n \to \infty}k^{3n}
f(\frac{x}{k^n})$ under suitable conditions on the function
$\varphi$ then obtain similar result to Corollary $2.2$ for $p<3.$
\end{rk}
\vskip .2in
\section {{\Large\bf} Stability of the quartic functional equation}
In this section, we prove the generalized Hyers-Ulam-Rassias
stability of the quartic functional equation $(1.6)$. Throughout
this section, we assume that $G$ is an additive group and $X$ is a
complete non-Archimedean space. Now before taking up the main
subject, given $f:G \to X,$ we define the difference operator
$$\Delta f(x,y)=f(kx+y)+f(kx-y)=k^2[f(x+y)+f(x-y)]+2k^2(k^2-1)f(x)-2(k^2-1)f(y)$$ for all
$x,y \in G$ and all $k \in \Bbb N.$ we consider the following
functional inequality:
$$\|\Delta f(x,y)\| \leq \psi(x,y)$$
for an upper bound $\psi:G \times G \to [0, \infty).$
\begin{thm}
Let $\psi:G \times G \to [0,\infty)$ be a function such that
$$\lim_{n \to \infty} \frac{\psi(k^nx,k^ny)}{|k|^{4n}}=0 \eqno(3.1)$$
for all $x,y \in G,$ $k \in \Bbb N$ and let for each $x \in G$ the
limit
$$\lim_{n \to \infty} \max \{\frac{\psi(k^jx,0)}{|k|^{4j}}:~0 \leq j < n \},\eqno(3.2)$$
denoted by $\tilde{\psi}(x),$ exists. Suppose that $f:G \to X$ is
a mapping satisfying $f(0)=0$ and
$$\|\Delta f(x,y)\| \leq \psi(x,y) \eqno(3.3)$$
for all $x,y \in G.$ Then there exists a quartic mapping $Q:G \to
X$ such that
$$\|Q(x)-f(x)\| \leq \frac{1}{|2k^4|} \tilde{\psi}(x) \eqno(3.4)$$
for all $x \in G,$ if $$\lim_{i \to \infty} \lim_{n \to \infty}
\max \{\frac{\psi(k^jx,0)}{|k|^{4j}}:~i \leq j < n+i \}=0$$ then
$Q$ is the unique quartic mapping satisfying $(3.4).$
\end{thm}
\begin{proof}
Putting $y=0$ in $(3.3),$ we get
$$\|\frac{1}{k^4}f(kx)-f(x)\| \leq \frac{1}{|2k^4|} \psi(x,0)\eqno(3.5)$$
for all $x \in G$ and all $k \in \Bbb N.$ Replacing $x$ by
$k^{n-1}x$ in $(3.5),$ we get
$$\|\frac{1}{k^{4n}}f(k^nx)-\frac{1}{k^{4(n-1)}}f(k^{n-1}x)\| \leq \frac{\psi(k^{n-1}x,0)}{|2k^{4n}|}\eqno(3.6)$$
for all $x \in G$ and all $k \in \Bbb N.$ It follows from $(3.6)$
and $(3.1)$ that the sequence $\{ \frac{f(k^nx)}{k^{4n}} \}$ is
Cauchy. Since $X$ is complete, we conclude that $\{
\frac{f(k^nx)}{k^{4n}} \}$ is convergent. Set $Q(x):=\lim_{n \to
\infty} \frac{f(k^nx)}{k^{4n}}.$ From the inequality (3.5) we use
iterative methods and induction on $n$ to prove our next relation:
$$\| \frac{f(k^nx)}{k^{4n}}-f(x)\| \leq  \frac{1}{|2k^4|} \max
\{\frac{\psi(k^ix,0)}{|k|^{4i}}:~ 0 \leq i < n \} \eqno(3.7)$$ for
all $n \in \Bbb N$ and all $x \in G.$ By taking $n$ to approach
infinity in $(3.7)$ and using $(3.2)$, we  obtain $(3.4).$ By
$(3.1)$ and $(3.3),$ we get
$$\|\Delta Q(x,y)\|=\lim_{n \to \infty}\frac{1}{|k^{4n}|}\|f(k^nx,k^ny)\| \leq \lim_{n \to \infty} \frac{\psi(k^nx,k^ny)}{|k|^{4n}}=0$$
for all $x,y \in G$ and all $k \in \Bbb N.$ Therefore the mapping
$Q:G \to X$ satisfies $(1.6).$ To prove the uniqueness property of
$Q$, let  $Q^{'}$ be  another quartic mapping satisfies  $(3.4),$
then
\begin{align*}
\|Q(x)-Q^{'}(x)\|&=\lim_{i \to
\infty}|k|^{-4i}\|Q(k^ix)-Q^{'}(k^ix)\|\\
& \leq \lim_{i \to \infty} |k|^{-4i} \max \{~
\|Q(k^ix)-f(k^ix)\|,\|f(k^ix)-Q^{'}(k^ix)\|~ \}\\
&\leq \frac{1}{|2k^4|} \lim_{i \to \infty} \lim_{n \to \infty}
\max \{ \frac{\psi(k^jx,0)}{|k|^{4j}}:~i \leq j < n+i \}\\
&=0.
\end{align*}
for all $x \in G$ and all $k \in \Bbb N.$ Therefore $Q=Q^{'}.$
This completes the proof of the uniqueness of $Q.$
\end{proof}
\begin{cor}
Let $\beta:[0,\infty) \to [0,\infty)$ be a function satisfying\\
$(i)$ $\beta(|k|t) \leq \beta(|k|) \beta(t)$ for all $t \geq 0,$\\
$(ii)$ $\beta(|k|) < |k|^4.$\\ Let $\delta > 0,$ let $G$ be a
normed space and let $f:G \to X$ fulfill the inequality
$$\|\Delta f(x,y)\| \leq \delta [\beta (\|x\|)+\beta (\|y\|)]$$ for all
$x,y \in G.$ Then there exists a unique quartic mapping $Q:G \to
X$ such that
$$\|f(x)-Q(x)\| \leq \frac{1}{|2k^4|} \delta \beta(\|x\|) \eqno(3.8)$$ for all $x \in G$ and all $k \in \Bbb N.$
\end{cor}
\begin{proof}
Putting $\psi(x,y):=\delta [\beta(\|x\|)+\beta(\|y\|)]$ in above
theorem. We have
$$\lim_{n \to \infty} \frac{\psi(k^nx,k^ny)}{|k|^{4n}} \leq \lim_{n \to \infty}(\frac{\beta(|k|)}{|k|^4})^{n} \psi(x,y)=0$$
for all $x,y \in G$ and all $k \in \Bbb N.$ We have
$$\tilde{\psi}(x)=\lim_{n \to \infty} \max \{ \frac{\psi(k^jx,0)}{|k|^{4j}}:~0 \leq j < n
\}=\psi(x,0)$$ and
$$\lim_{i \to \infty} \lim_{n \to \infty} \max \{ \frac{\psi(k^jx,0)}{|k|^{4j}}:~ i \leq j < n+i \}= \lim_{i \to \infty} \frac{\psi(k^ix,0)}{|k|^{4i}}=0$$
for all $x \in G$ and all $k \in \Bbb N.$
\end{proof}
\begin{rk}
The classical example of the function $\beta$ is the mapping
$\beta(t)=t^p$ for all $t \in [0, \infty),$ where $p>4$ with the
further assumption that $|k|<1.$
\end{rk}
\begin{rk}
We can formulate similar statements to Theorem $3.1$ in which we
can define the sequence $Q(x):=\lim_{n \to \infty}k^{4n}
f(\frac{x}{k^n})$ under suitable conditions on the function $\psi$
then obtain similar result to Corollary $3.2$ for $p<4.$
\end{rk}

{\small

}

\begin{thebibliography}{99}

\bibitem{Ar-Be} L. M. Arriola and W. A. Beyer, Stability of the Cauchy functional equation over p-adic
fields, {\it Real Analysis Exchange.} {\bf 31} (2005/2006),
125–-132.

\bibitem{Bo-Fo} C. Borelli and G. L. Forti, On a general Hyers-Ulam stability
result, {\it Internat. J. Math. Math. Sci.} {\bf 18} (1995),
229-–236.

\bibitem{Ch} P. W. Cholewa, Remarks on the stability of functional equations, {\it Aequationes Math.} {\bf 27} (1984), 76--86.

\bibitem{Cz} S. Czerwik, On the stability of the quadratic mapping in normed
spaces. {\it Abh. Math. Sem. Univ. Hamburg.} {\bf 62} (1992),
59-–64.

\bibitem{Ga} Z. Gajda, On stability of additive mappings. {\it Internat. J. Math.
Math. Sci.} {\bf 14} (1991), 431-–434.

\bibitem{G} P. G\v avruta, A generalization of the Hyers-Ulam-Rassias stability of approximately additive mappings, {\it J. Math. Anal. Appl.} {\bf 184}  (1994), 431--436.

\bibitem {Gr} A. Grabiec, The generalized Hyers-Ulam stability of a class of functional equations, {\it Publ. Math. Debrecen} {\bf 48} (1996), 217--235.

\bibitem {Go} F. Q. Gouv\t ea, p-adic Numbers. {\it Springer-Verlag, Berlin.} 1997.

\bibitem {He} K. Hensel, \"Uber eine neue Begr\" undung der Theorie der
algebraischen Zahlen. {\it Jahresber. Deutsch. Math. Verein.} {\bf
6} (1897), 83-–88.

\bibitem{Hy} D. H. Hyers, On the stability of the linear functional equation, {\it Proc. Nat. Acad. Sci. USA} {\bf 27} (1941), 222--224.

\bibitem{Hy-Is-Ra} D. H. Hyers, G. Isac and Th. M. Rassias, Stability of Functional Equations in Several Variables, {\it Birkh\"{a}uer, Basel.} 1998.

\bibitem{Ju} S. M. Jung, Hyers-Ulam-Rassias Stability of Functional Equations in Mathematical
Analysis. {\it Hadronic Press lnc. Palm Harbor, Florida.} 2001.

\bibitem{Ju-Ki} K. W. Jung and H. M. Kim, The generalized Hyers-Ulam-Rassias
stability of a cubic functional equation, {\it J. Math. Anal.
Appl.} {\bf 274} (2002), no. 2, 267-–278.

\bibitem{Kh} A. Khrennikov, Non-Archimedean Analysis, Quantum Paradoxes,
Dynamical Systems and Biological Models. {\it Kluwer Academic
Publishers, Dordrecht.} 1997.

\bibitem{Mo-Ra} M. S. Moslehian and Th. M. Rassias, Stability of functional
equations in non-Archimedean spaces, {\it Applicable Analysis and
Discrete Mathematics.} {\bf 1} (2007), 325–-334.

\bibitem{Na} A. Najati, On the stability of a quartic functional equation, {\it J.
Math. Anal. Appl.} {\bf 340} (2008), no. 1, 569--574.

\bibitem{Pa1} C. Park, Generalized quadratic mappings in several variables.
{\it Nonlinear Anal. TMA.} {\bf 57} (2004), 713–-722.

\bibitem{Pa2} C. Park, On the stability of the quadratic mapping in Banach
modules. {\it J. Math. Anal. Appl.} {\bf 276} (2002), 135–-144.

\bibitem{Pa-Ba} W. G. Park and J. H. Bae, On the stability a bi-quartic functional
equation, {\it Nonlinear Anal.} {\bf 62} (2005), no. 4, 643-–654.

\bibitem {R1} Th. M. Rassias (Ed.), {  Functional Equations and
 Inequalities,}{\it Kluwer
Academic, Dordrecht,} 2000.


\bibitem {TRa1} Th. M. Rassias, On the stability of the linear mapping in Banach spaces, {\it Proc. Amer. Math. Soc.} {\bf 72} (1978), 297--300.

\bibitem {TRa2} Th. M. Rassias, On the stability of the quadratic functional equation and its applications.
{\it Studia Univ. Babes-Bolyai.} {\bf 43} (1998), 89-–124.

\bibitem {Ro} A. M. Robert, A Course in p-adic Analysis. {\it Springer–Verlag, New
York.} 2000.

\bibitem {Sk} F. Skof, Propriet\' a localie approssimazione dioperatori, {\it Rend.
Sem. Mat. Fis. Milano.} {\bf 53} (1983), 113-–129.

\bibitem {Ul} S. M. Ulam, A Collection of Mathematical Problems, {\it Interscience Tracts in Pure and Applied Mathematics, Interscience Publisher, New York,} 1960.

\bibitem {Vl-Vo-Ze} V. S. Vladimirov, I. V. Volovich and E. I. Zelenov, p-adic Analysis
and Mathematical Physics. {\it World Scientific.} 1994.

\end{thebibliography}
\end{document}